\documentclass[11pt]{amsart}

\usepackage{epsfig,overpic}
\usepackage{verbatim}
\usepackage[usenames,dvipsnames,svgnames,table]{xcolor}
\usepackage[hyphens]{url}
\usepackage[pagebackref,linktocpage=true,colorlinks=true,linkcolor=Blue,citecolor=BrickRed,urlcolor=RoyalBlue]{hyperref}
\usepackage[msc-links,abbrev]{amsrefs}
\usepackage{amsmath,amsthm,amssymb,marginnote}
\usepackage{enumerate}
\usepackage[T1]{fontenc}
\usepackage[all]{xy}\usepackage[labelformat=empty]{caption}

\newtheorem{thm}{Theorem}[section]
\newtheorem{lem}[thm]{Lemma}

\newtheorem{prop}[thm]{Proposition}

\theoremstyle{definition}

\newtheorem{note}[thm]{Note}

\theoremstyle{remark}

\newcommand{\R}{\mathbf{R}}
\newcommand{\ol}[1]{{\overline #1}}

\newcommand{\C}{\mathcal{C}}

\newcommand{\D}{\mathcal{D}}

\renewcommand{\S}{\mathbf{S}}

\renewcommand{\d}{\partial} 
\renewcommand{\tilde}{\widetilde} 
\renewcommand{\H}{\textbf{H}}

\DeclareMathOperator{\inte}{int}

\DeclareMathOperator{\Conf}{Con}
\DeclareMathOperator{\Sym}{Sym}
\DeclareMathOperator{\Fix}{Fix}
\DeclareMathOperator{\Isom}{Iso}

\begin{document}


\title[Point selections from Jordan domains]{Point selections from Jordan domains\\in Riemannian Surfaces}

\author{Igor Belegradek}
\address{School of Mathematics, Georgia Institute of Technology,
Atlanta, GA 30332}
\email{ib@math.gatech.edu}
\urladdr{www.math.gatech.edu/~ib}

\author{Mohammad Ghomi}
\address{School of Mathematics, Georgia Institute of Technology,
Atlanta, GA 30332}
\email{ghomi@math.gatech.edu}
\urladdr{www.math.gatech.edu/~ghomi}

\date{Last revised \today.}
\thanks{The research of I.B.  was supported by the Simons Foundation grant 524838. The research of M.G. was supported  by NSF Grant DMS--2202337.}

\begin{abstract}
Using fiber bundle theory and conformal mappings, we continuously select a point from the interior of Jordan domains in Riemannian surfaces. This selection can be made equivariant under isometries, and take on prescribed values such as the center of mass when the domains are convex. Analogous results for conformal transformations are obtained as well. It follows that the space of Jordan domains in surfaces of constant curvature admits an isometrically equivariant strong deformation retraction onto the space of round disks. Finally we 
develop a canonical procedure for selecting points from planar Jordan domains.
\end{abstract}

\maketitle



\section{Introduction}
How can one  continuously select a point from the interior of Jordan domains in Euclidean plane $\R^2$? Eugenio Calabi asked this question from the second-named author in 1995. In \cite{belegradek-ghomi2020} we provided an answer for smooth domains. Here we extend that result to the topological category, and develop some of its applications.

Let $M$ be a connected Riemannian $2$-manifold and $B^2$ be the closed unit disk in 
$\R^2$. We say $D\subset M$ is a \emph{Jordan domain} if it forms the image of a continuous injective map $B^2\to M$, called a parametrization of $D$. The space of Jordan domains $\D(M)$ is the collection of all Jordan domains in $M$ with the topology induced by their parametrizations, i.e., a pair of domains are close provided that they admit parametrizations which are $\C^0$-close. A \emph{center} on $A\subset\D(M)$ is a continuous map $c\colon A\to M$ such that $c(D)\in \inte(D)$ for all $D\in A$, where int denotes the interior. Let $\Isom(M)$ be the group  of isometries of $M$, and $G\leq \Isom(M)$ be a closed subgroup.
We say that $c$ is  \emph{G-equivariant}  if for any $\rho\in G$ and $D\in A$, $\rho(D)\in A$, i.e., $A$ is \emph{$G$-invariant}, and $ \rho(c(D))=c(\rho(D))$. We show:

\begin{thm}\label{thm:main}
Any $G$-equivariant center on a closed $G$-invariant subset of
$\D(M)$ may be extended to a $G$-equivariant center  on $\D(M)$.
\end{thm}

There are a number of classical point selections,  such as the center of mass or Steiner point, from the interior of convex domains in $\R^2$ \cite{KMT1991} \cite[Chap. 12]{moszynska:book} \cite[Sec. 5.4.1]{schneider2014},  or other Riemannian surfaces \cite{grove-karcher1973,grove1976,karcher1977}. By Theorem \ref{thm:main}, any of these selections may be extended to a $G$-equivariant center on  nonconvex Jordan domains. This result had been established earlier for $\C^1$ Jordan domains \cite{belegradek-ghomi2020} via equivariant topology. Here we use conformal mappings to enhance those  techniques, as described in Sections \ref{sec:general} and \ref{sec:equiv} below. We will also obtain analogues of Theorem \ref{thm:main} for conformal transformations (see Theorem \ref{cor:conformal} and Note \ref{note:conformal}). Some applications will be discussed in Section \ref{sec:application}.  In particular we show that $\D(\R^2)$ is equivariantly contractible with respect to the orthogonal group $O(2)$, and $\D(\S^2)$ admits an isometrically equivariant strong deformation retraction  onto the space of hemispheres (Theorem \ref{cor:contract2}). Finally in Section \ref{sec:canonical} we use the notion of reach in the sense of Federer to develop a canonical procedure for constructing centers for smooth Jordan domains in $\R^2$ (Theorem \ref{thm:C2}).

Our central results here may be framed in terms of continuous selections from multivalued maps, which is a well-known area in functional analysis \cite{repovs-semenov1998, repovs-semenov2002, repovs-semenov2014} originating with Michael~\cite{michael1956a, michael1956b, michael1959}. In those results, however, one usually assumes that the maps take their value from subsets which satisfy some notion of convexity. The closest antecedent for this work and the authors earlier paper \cite{belegradek-ghomi2020}, which applies to nonconvex objects, is due to Pixley \cite{pixley1976} who solved a problem of Bing on point selection from curve segments.  We should also mention the notion of \emph{geodesic center} for polygonal domains in $\R^2$ \cite{asano-toussaint1987, pollack-Sharir-Rote1989},
which is well-known in computational geometry; however, this point does not always lie in the interior, and is well-defined only when the boundary is rectifiable. Selection theorems have important applications in optimal control and differential inclusion \cite{parathasarathy1972}.

\begin{note}
Daniel Asimov  \cite{asimov-email} reports that around 1977 he discussed with Calabi an intrinsic version of the problem we study here, which may be described as follows. Let $B^n$ denote the closed unit ball in Euclidean space $\R^n$, and $\textup{Met}(B^n)$ be the space of smooth Riemannian metrics on $B^n$. 
A diffeomorphism $f\colon B^n\to B^n$ is an isometry with respect to $g\in \textup{Met}(B^n)$, or a \emph{$g$-isometry}, provided that $f^*g=g$, where 
$(f^*g)_p(v,w):=g_{f(p)}(df_p(v),df_p(w))$, for $p\in B^n$ and $v$, $w\in\R^n$. Does there exist a continuous map $c\colon \textup{Met}(B^n)\to B^n$ which is equivariant under isometries, i.e., for every $g\in \textup{Met}(B^n)$ and $g$-isometry $f\colon B^n\to B^n$, $f(c(g))=c(g)$? A necessary condition is that $n\leq 5$ because in higher dimensions there exist groups of diffeomorphisms acting on $B^n$ without a common fixed point \cite{bak-morimoto2005}.
\end{note}

\section{General Centers}\label{sec:general}
We start by establishing Theorem \ref{thm:main} in the case where $G$ is  trivial:

\begin{prop}\label{prop:center}
Any center on a closed subset of $\D(M)$ may be extended to a center on $\D(M)$.
\end{prop}

 The main idea here is to frame the problem in terms of the existence of a section of a fiber bundle. The principal fact we need  is:

\begin{lem}[Palais \cite{palais1966}, Thm. 9]\label{lem:palais}
If a locally trivial fiber bundle has a metrizable base and metrizable contractible fibers, then it admits a global section. Furthermore, any partial section defined on a closed subset of the base may be extended to a global section. 
\end{lem}

For example, the above lemma applies to any locally trivial disk bundle over a metric space. To utilize Lemma \ref{lem:palais}, we  first establish the following fact. Let $\textup{Emb}^0(B^2,M)$ denote the space of injective continuous maps $B^2\to M$, with its standard $\C^0$-topology as a subspace of continuous maps $\C^0(B^2,M)$.
\begin{lem}\label{lem:metrizable}
$\D(M)$ is metrizable. 
\end{lem}
\begin{proof}
Let $\delta\colon M\times M\to\R$ be the metric of $M$.
Given $D_1$, $D_2\in \D(M)$, set 
\begin{equation}\label{eq:metric}
d(D_1,D_2):= \inf\Big\{\sup_{p\in B^2}\delta\big(f_1(p), f_2(p)\big)\;\big|\; f_i\in \textup{Emb}^0(B^2,M), f_i(B^2)=D_i\Big\}.
\end{equation}
It is not difficult to check that $d$ is a metric on $\D(M)$. 
\end{proof}
Next we associate a fiber bundle to $\D(M)$ whose total space is generated by interior points of the domains:
$$
\mathcal{E}(M):=\big\{(D,p)\mid D\in \D(M), \;p\in\inte(D)\big\}.
$$
Let $\pi\colon \mathcal{E}(M)\to \D(M)$ be the projection given by
$
\pi(D,p):=D.
$
 Then each fiber $\pi^{-1}(D)$ is homeomorphic to $\inte(B^2)$, which is contractible. To show that 
$\pi$ is locally trivial, we need the following fact.  A \emph{model plane} is either the Euclidean plane $\R^2$, the unit sphere $\S^2$, or the hyperbolic plane $\textbf{H}^2$. Let $o$ denote the origin of $\R^2$.

\begin{lem}[Carath\'{e}odory]\label{lem:Riemann}
Let $M$ be a  model plane. For any domain $D\in \D(M)$, point $p\in\inte(D)$, and vector $u\in T_pM\setminus\{0\}$, there exists a unique continuous map 
$$
f=f_{D,p,u}\in \textup{Emb}^0(B^2,M)
$$
 with $f(B^2)=D$, such that $f$ is conformal on $\inte(B^2)$, $f(o)=p$, and $df_o((1,0))$ is parallel to $u$. Furthermore, if  $D_i\in \D(M)$ is a sequence of domains converging  to $D$,  then $f_{D_i,p,u}$ converges to $f$.
\end{lem}
\begin{proof}
The first statement is Carath\'{e}odory's refinement of Riemann mapping theorem \cite[Thm 5.1.1]{krantz:gft}, which ensures that conformal maps from the interior of $B^2$ to the interior of a Jordan domain $D\subset M$ extend to a homeomorphism $B^2\to D$. The second statement is a consequence of Carath\'{e}odory's kernel theorem, and follows quickly from \cite[Thm. 2.11]{Pommerenke:book}.
\end{proof}

We now show that
$\pi$ is locally trivial, which will conclude the proof of Proposition \ref{prop:center}.
Pick $D_0\in\D(M)$, and fix $p_0\in \inte(D_0)$, $u_0\in T_{p_0} M\setminus\{0\}$. Let $U$ be a neighborhood of $D_0$ in $\D(M)$ so small that $p_0$ belongs to the interior of every domain $D\in U$. There exists then, for every $D\in U$, a canonical homeomorphism $f_{D}:=f_{D,p_0,u_0}\colon B^2\to D$ given by Lemma \ref{lem:Riemann}. Now the map
$$
\pi^{-1}(U)\;\;\ni\;(D,p)\overset{\phi}{\longmapsto}\big(D,f_D^{-1}(p)\big)\;\;\in\;\; U\times \inte(B^2)
$$
yields the desired trivialization. In particular note that  $D\mapsto f_D$ is continuous by Lemma \ref{lem:Riemann}, which ensures that $D\mapsto f_D^{-1}$ and consequently $\phi$  are continuous.

\section{Equivariant Centers}\label{sec:equiv}
Here we modulate the fiber bundle approach in the last section by  the action of $G\le\Isom(M)$  to  
complete the proof of Theorem \ref{thm:main}.
Similar to \cite{belegradek-ghomi2020} this involves stratifying $\D(M)$ by its orbit types under $G$, and extending the center inductively to each stratum; however, here we can implement this procedure more directly via conformal maps. A topological group $G$ \emph{acts properly} on $M$ provided that for any compact set $K\subset M$, the collection of $g\in G$ such that $g(K)\cap K\neq\emptyset$ is compact \cite[Prop. 9.12]{lee2003}. It is well-known that $\Isom(M)$ acts properly on $M$ \cite[Thm. I.4.7]{kobayashi-nomizu1996a}, but this is not  the case for the group $\Conf(M)$ of conformal transformations:

\begin{lem}\label{lem:proper}
If $G\leq\Conf(\R^2)$ acts properly on $\R^2$, then $G\leq\Isom(\R^2)$.
\end{lem}
\begin{proof}
Let $g\in G$. Then $g(x)=r\rho(x)+b$, where $r>0$, $\rho\in O(2)$, and $b\in\R^2$. Note that $g\in\Isom(\R^2)$ if and only if $r=1$. If $r\neq 1$, then $I-r\rho$ is invertible, where $I$ is the identity transformation. Consequently
$g$ fixes $x_0:=(I-r\rho)^{-1}(b)$, which yields $g(x)=x_0+r\rho(x-x_0)$.  So $g^n(x)=x_0+r^n\rho^n(x-x_0)$, and $|g^n(x)-x_0|=r^n$ for integers $n$. Thus  the set $\{g^n(x_1)\}$ is unbounded for any $x_1\in\R^2\setminus\{x_0\}$. Hence $\{g^n\}$ is not a compact subset of $G$, since $\{g^n(x_1)\}$ is the image of $\{g^n\}$ under the continuous mapping $\{g^n\}\to\R^2$, given by $x_1\mapsto g^n(x_1)$. On the other hand,
if $K:=x_0+B^2$, then $g^n(K)=x_0+r^n B^2$. So
$g^n(K)\cap K=(r^n B^2)\cap B^2\neq\emptyset$, which is not possible if $G$ acts properly. Hence $r=1$, which means $g\in\Isom(\R^2)$. 
\end{proof}

We may assume that $M$ is simply connected after replacing it with its universal Riemannian cover,
and replacing $G$ with the group of all lifts of elements of $G$ to the covering space, because any equivariant center on the universal cover descends to an equivariant center on $M$.
Then $M$ (which may not be complete) is conformal to a complete Riemannian manifold $\ol M$ \cite[Thm. 1]{nomizu1961}
which we may assume to be a model plane by the uniformization theorem. Now
$G\leq\Conf(\ol M)$. We claim that  $G\leq\Isom(\ol M)$. This is immediate when 
$\ol M=\textbf{H}^2$,  since $\Conf(\textbf{H}^2)=\Isom(\H^2)$. If $\ol M=\S^2$,
then $\Isom(M)$ is compact. So $G$ is a compact subgroup of $\Conf(\S^2)$. Hence $G$ is conjugate 
to a subgroup of $\Isom(\S^2)$ (which means that it  is an isometry group of 
an isometric copy of $\S^2$). If $\ol M=\R^2$, note that since $G\leq \Isom(M)$, it  acts properly on $M$, and therefore on $\ol M$. Hence $G\leq\Isom(\ol M)$ by Lemma \ref{lem:proper}.
In short, after replacing $M$ with $\ol M$, we may assume that $M$ is a model plane. To 
prove Theorem \ref{thm:main} it remains to show:

\begin{prop}\label{prop:model}
Let $M$ be a model plane. Then any $G$-equivariant center on a closed $G$-invariant subset of $M$ may be extended to a $G$-invariant center on $M$.
\end{prop}

For the rest of this section we assume that $M$ is a model plane, and $G\leq\Isom(M)$ is a closed subgroup. For any $D\in\D(M)$, let $\Sym(D)\leq G$ be the maximal subgroup  which maps $D$ onto itself,   
and  $\Fix(D)$ be the set of points in $D$ which are fixed by all
$\rho\in\Sym(D)$. We say $\rho\in\Sym(D)$ is a \emph{rotation} if it preserves the orientation, and is a \emph{reflection} if it reverses the orientation of $M$. Let us record that:

\begin{lem}\label{lem:dim}
For any $D\in \D(M)$, $\Sym(D)$ is either (i) the trivial group, (ii) a group of order $2$ generated by a reflection, (iii)  a group which contains a cyclic subgroup of rotations. In these cases $\Fix(D)$ is equal respectively to (i) $D$, (ii) a simple geodesic arc with relative interior in $\inte(D)$ and end points on $\d D$, or (iii) a single point in the interior of $D$.
\end{lem}

So $\D(M)$ may be partitioned into subsets
\begin{gather*}
\D^k(M):=\big\{D\in \D(M)\mid\dim\big(\Fix(D)\big)=k\big\},
\end{gather*}
for $k=0$, $1$, $2$.
Let  $\overline{\D^k}(M)$ denote the closure of $\D^k(M)$ in $\D(M)$. 

\begin{lem}\label{lem:Ds}
For all $0\leq k\leq 2$, $\overline{\D^k}(M)\subset\bigcup_{\ell=0}^k\D^\ell(M)$. Furthermore, if $k\neq 1$, then $\overline{\D^k}(M)=\bigcup_{\ell=0}^k\D^\ell(M)$.
\end{lem}
\begin{proof}
Let $D\in\overline{\D^k}(M)$. Then there exists a sequence $D_i\in \D^k(M)$ converging to $D$. Let $\rho_i\in\Sym(D_i)$, and $d$ be the metric on $\D(M)$ given by \eqref{eq:metric}. 
Then 
$$
d( \rho_i(D), D)\leq d( \rho_i(D),D_i)+d(D_i, D)=2d(D, D_i),
$$
since
$
d( \rho_i(D),D_i)=d(\rho_i(D),  \rho_i(D_i))=d(D, D_i).
$
 So $\rho_i(D)\to D$. Each $\rho_i$ fixes some point $p_i\in D_i$ by Brouwer's fixed point theorem. Since $D_i\to D$, the sequence $p_i$ is bounded.  Since $\Isom(M)$ is locally compact, and $G$ is closed, it follows that $\rho_i$
 converges to some $\rho\in G$, after passing to a subsequence. So $\rho_i(D)\to\rho(D)$. Now note that
 \begin{eqnarray}\label{eq:rhos}
 d(\rho(D), D)&\leq& d(\rho(D),\rho_i(D))+ d(\rho_i(D),\rho_i(D_i))+d(\rho_i(D_i),D)\\\notag
 &=& d(\rho(D),\rho_i(D))+ d(D,D_i)+d(D_i,D).
 \end{eqnarray}
So $\rho(D)=D$ which yields $\rho\in\Sym(D)$.  Thus $\Sym(D_i)$ converge to some subgroup of $\Sym(D)$. Since reflections converge to a reflection, and rotations to a rotation, it follows from Lemma \ref{lem:dim} that $D\in\D^\ell(M)$ for some $\ell\leq k$. So $\overline{\D^k}(M)\subset\bigcup_{\ell=0}^k\D^\ell(M)$.
The reverse inclusion holds trivially for $k=0$. For $k=2$ note that any domain in $\D^0(M)$ or $\D^1(M)$ may loose all it symmetries by an arbitrarily small perturbation, and thus belongs to $\overline{\D^2}(M)$.
 \end{proof}

The exclusion of $k=1$ in the above lemma is necessary, because there are domains $D\in\D^0(\R^2)$, such as a generic parallelogram, which have no bilateral symmetry and hence cannot be a limit of domains in $\D^1(\R^2)$.
The arguments in the last lemma also yield the next observation. For $k=0$, $1$ here, continuity  is with respect to the Hausdorff distance between compact subsets of $M$.

\begin{lem}\label{lem:fix}
The mapping $D\mapsto\Fix(D)$ is continuous on $\D^k(M)$.
\end{lem}
\begin{proof}
If $k=2$, then $\Fix(D)=D$ and there is nothing to prove. So assume $k\leq 1$.
Let $D_i\in \D^k(M)$ be a sequence of domains which converges to $D\in \D^k(M)$. Choose $p_i\in\Fix(D_i)$. Then $p_i$ converge to a point $p\in D$, after passing to a subsequence. Let $\rho_i\in\Sym(D_i)$. Then $\rho_i(p_i)=p_i$. As discussed in the proof of Lemma \ref{lem:Ds}, $\rho_i$ converge to $\rho\in\Sym(D)$. A computation similar to \eqref{eq:rhos} shows that $\rho(p)=p$. 
If $k=0$, $p_i=\Fix(D_i)$ and we may assume that $\rho_i$ are rotations by an angle which is uniformly bounded away from $0$ and $2\pi$. Then $\rho$ will be a nontrivial rotation. Hence $p=\Fix(D)$ as desired.
Next suppose that $k=1$. Then we may assume that $\rho_i$ are reflections, which yields that so is $\rho$. Consequently, $p\in \Fix(D)$. So $\Fix(D_i)$ converge to a subset $S$ of $\Fix(D)$. 
Since $\Fix(D_i)$ are connected, so is $S$. Furthermore,  if we choose $p_i\in\partial D$, then $p\in\partial D$. So $S$ contains the end points of $\Fix(D)$. Hence $S=\Fix(D)$, which completes the proof.
\end{proof}

If a center $c$ on $A\subset\D(M)$ is $G$-equivariant, then $c(D)\in\Fix(D)$ for all $D\in A$. Thus the last lemma establishes a necessary requirement for the existence of $G$-equivariant centers, which we proceed to construct as follows. Let 
$$
\mathcal{E}^k(M):=\big\{(D,p)\mid D\in\D^k(M)\,,\; p\in\textup{relint}\big(\Fix(D)\big)\big\},
$$
where relint stands for relative interior.
Next let 
$
[D]:=\{\rho(D)\mid\rho\in G\}
$
be the $G$-orbit of $D\in\D^k(M)$,  and
$
[\D^k(M)]:=\D^k(M)/G
$
be the corresponding orbit space.
Also let
$
[(D,p)]:=\{(\rho(D),\rho(p))\mid \rho\in G\}
$
be the $G$-orbit of $(D,p)\in\mathcal{E}^k(M)$, and set
$
[\mathcal{E}^k(M)]:=\mathcal{E}^k(M)/G.
$
The bundle map $\pi\colon \mathcal{E}^k(M)\to \D^k(M)$, given by $\pi((D,p)):=D$, descends to $\tilde\pi \colon [\mathcal{E}^k(M)]\to [\D^k(M)]$
given by $\tilde \pi([(D, p)]):=[D]$. So $\tilde\pi^{-1}([D])$ is homeomorphic to $\textup{relint}(B^k)$.
We also have the natural orbit projections $(D, p)\mapsto[(D, p)]$ and $D\mapsto[D]$, which form the following commuting diagram:
\begin{equation*}
\label{form: cd}
\xymatrix{
\mathcal{E}^k(M) \ar[r]^\pi
\ar[d]&  \D^k(M)\ar[d]  \\
[\mathcal{E}^k(M)]\ar[r]^{\tilde \pi}& [\D^k(M)]}
\end{equation*}
By \cite[Lem. 3.1]{belegradek-ghomi2020} the action of $G$ on $\D^k(M)$ is \emph{proper in the sense of Palais} \cite[Def. 1.2.2]{palais1961}, i.e., each $D\in\D^k(M)$ has a neighborhood $V$ such that for every $D'\in\D^k(M)$ there is a neighborhood $U$ with the property that $\{\rho\in G\mid \rho(U)\cap V\neq\emptyset\}$ has compact closure. Furthermore, since by assumption $G$ is closed, it is a Lie group by Cartan's closed subgroup theorem \cite[Thm. 20.10]{lee2003}. Hence the slice result of Palais \cite[Prop. 2.3.1] {palais1961} for orbits of proper actions applies to $[\D^k(M)]$, which leads to the next observation.

\begin{lem}\label{lem:trivial}
The bundle map $\tilde\pi \colon [\mathcal{E}^k(M)]\to [\D^k(M)]$ is locally trivial.
\end{lem}
\begin{proof}
If $k=0$, then $\tilde\pi$ is injective, and by Lemma \ref{lem:fix} is a homeomorphism. 
Next assume  that $k=1$. Let $D_0\in\D(M)$, and $\Gamma$ be the complete geodesic containing $\Fix(D_0)$. By Lemma \ref{lem:fix}, we may choose a neighborhood $U$ of $D_0$ in $\D^k(M)$ so small that $\Fix(D)$ lies in a tubular neighborhood of $\Gamma$ for all $D\in U$. Then $\Fix(D)$ may be projected into a segment $\overline{\Fix(D)}\subset\Gamma$, and canonically mapped to $\Fix(D_0)$, i.e., by a translation in $\Gamma$ so that the midpoint of $\overline{\Fix(D)}$ coincides with that of $\Fix(D_0)$, and then a dilation.
Finally suppose that $k=2$.
As we showed at the end
of Section \ref{sec:general}, by Lemma \ref{lem:Riemann}, there exists a neighborhood $U$ of $D_0$ in $\D^k(M)$ such that 
for every $D\in U$ there exists a homeomorphism $f_D\colon B^2\to D$ depending continuously on $D$. Let $[U]\subset [\D^k(M)]$ be the corresponding neighborhood of $[D_0]$, i.e., the collection of $[D]$ for all $D\in U$. 
By \cite[Prop. 2.3.1]{palais1961} there exists a local slice of $[\D^k(M)]$ through $D_0$, i.e., 
a continuous map $\sigma\colon [U]\to U$ such that $\sigma([D])\in[D]$, and $\sigma([D_0])=D_0$.
Then the mapping 
$$
\tilde\pi^{-1}([U])\;\;\ni\;\;[(D,p)]\longmapsto\Big([D],\, f_{\sigma([D])}^{-1}\big(\sigma([p])\big)\Big)\;\;\in\;\;[U]\times\inte(B^2),
$$
is the desired trivialization, where $\sigma([p])$ denotes the element of $[p]$ in $\sigma([D])$, i.e., $\sigma([p]):=\rho(p)$ where 
$\rho\in G$ is the isometry with $\rho(D)=\sigma(D)$ ($\rho$  is unique because $G$ acts freely on $\D^2(M)$).
\end{proof}

As noted in \cite[Rem. 4.14]{belegradek-ghomi2020}, the $\C^1$-regularity assumption in \cite{belegradek-ghomi2020} was needed only to prove the above lemma. Now that we have established this fact, via Carath\'{e}odory's theorem, the arguments in \cite{belegradek-ghomi2020} complete the proof of Theorem \ref{thm:main}; however, we include a shorter and self-contained argument below. Let $[\D(M)]:=\D(M)/G$.

\begin{lem}\label{lem:DandE}
$[\D(M)]$, $[\D^k(M)]$ and $[\mathcal{E}^k(M)]$ are metrizable.
\end{lem}
\begin{proof}
A locally trivial fiber bundle with a metrizable base and fiber is metrizable \cite{etter-griffin1954}. So, by Lemma \ref{lem:trivial}, $[\mathcal{E}^k(M)]$ is metrizable if $[\D^k(M)]$ is metrizable. Since $[\D^k(M)]\subset [\D(M)]$, it suffices to check that $[\D(M)]$ is metrizable.
By a result of Palais \cite[Thm. 4.3.4]{palais1961}, $[\D(M)]$ is metrizable, provided that $\D(M)$ is metrizable and separable.
By Lemma \ref{lem:metrizable}, $\D(M)$ is metrizable.  To see that $\D(M)$ is separable, note that it may be identified with
$\mathrm{Emb}^0(B^2, M)/\mathrm{Homeo}(B^2)$. So it suffices to check that  $\mathrm{Emb}^0(B^2, M)$ is separable, because  continuous image of a separable space is separable \cite[Thm. 16.4a]{willard2004}. But $\mathrm{Emb}^0(B^2, M)$ is a subspace of $\C^0(B^2, M)$, which is second countable \cite[Thm. XII.5.2]{dugundji1966}. Hence $\mathrm{Emb}^0(B^2, M)$ is second countable, and therefore separable \cite[Thm. 30.2]{munkres2000}.
\end{proof}

Using the last three observations, we can now establish the main step in the proof of Proposition \ref{prop:model}:

\begin{lem}\label{lem:D2}
Any $G$-equivariant center on a closed $G$-invariant subset of $\D^k(M)$ 
 may be extended to a $G$-equivariant center on $\D^k(M)$.
\end{lem}
\begin{proof}
Let $A$ be a closed $G$-invariant subset of $\D^k(M)$, and $c$ be a $G$-equivariant center on $A$. 
Define $C\colon A\to \mathcal{E}^k(M)$, by $C(D):=(D,c(D))$. Then $C$ descends to  the map $\tilde C\colon A/ G\to[\mathcal{E}^k(M)]$ given by $\tilde C([D])=[C(D)]$. Since $A$ is $G$-invariant and closed, $A/G$ is closed. So, by Lemmas \ref{lem:palais} and \ref{lem:trivial}, $\tilde C$ can be extended to $[\D^k(M)]$.
Now for $D\in \D^k(M)$ we define $\tilde  c(D)$ as the point in $\Fix(D)$ such that
$$
[(D,\tilde  c(D))]=\tilde C([D]).
$$
Note that $\tilde  c(D)$ is indeed unique, because if
 $[(D, x)]=[(D, \tilde  c(D))]$, then $x= \rho(\tilde  c(D))$ for some $\rho\in\Sym(D)$. But $\rho(\tilde  c(D))=\tilde  c(D)$ since $\tilde  c(D)\in\Fix(D)$. So $x=\tilde c(D)$.
Next, to see that $\tilde  c$ is $G$-equivariant, note that
$$
[(\rho(D),\tilde  c(\rho(D))]=\tilde C([\rho(D)])= \tilde C([D])=[(D,\tilde  c(D))]=[(\rho(D),\rho(\tilde  c(D)))].
$$
So $\rho(\tilde  c(D))=\rho'(\tilde  c(\rho(D)))$ for some $\rho'\in\Sym(\rho(D))$. But $\tilde  c(\rho(D))\in\Fix(\rho(D))$.
So $\rho'(\tilde  c(\rho(D)))=\tilde  c(\rho(D))$, and we conclude that $\rho(\tilde  c(D)))=\tilde  c(\rho(D))$.
Next we check that $\tilde  c=c$ on $A$. If $D\in A$, then $[(D, \tilde  c(D))]=\tilde C([D])=[C(D)]=[(D, c(D))]$. So $\tilde  c(D)=\rho(c(D))$ for some $\rho\in\Sym(D)$. Since $c$ is equivariant, $\rho(c(D))=c(\rho(D))=c(D)$. So $\tilde  c(D)=c(D)$ for $D\in A$.
Finally we check that $\tilde  c$ is continuous. Let $D_i\in\D^k(M)$ converge to $D$. By Lemma \ref{lem:fix}, $\tilde  c(D_i)$ converges to some point $x\in\Fix(D)$, after passing to a subsequence. So $[(D_i, \tilde  c(D_i))]$ converges to $[(D,x)]$. But $[(D_i,\tilde  c(D_i))]=\tilde C([D_i])$ which converges to $\tilde C([D])=[(D,\tilde  c(D))]$. By Lemma \ref{lem:DandE}, $[\mathcal{E}^k(M)]$ is Hausdorff. So $[(D_i, \tilde  c(D_i))]$ must have a unique limit, which yields $[(D,x)]=[(D,\tilde  c(D))]=\tilde C([D])$. Thus $x=\tilde  c(D)$ by the uniqueness of $\tilde  c(D)$.
\end{proof}

The last fact we need will ensure that the extensions given by the last lemma may  join each other continuously across  $\D(M)$. We establish this fact via the equivariant version of Tietze's extension theorem which was first proved by Gleason \cite{gleason1950} and later generalized by Feragen \cite{feragen2010}. The original Tietze's theorem states that continuous functions on closed subsets of a normal topological space extend to the entire space.

\begin{lem}\label{lem:last}
Let $A\subset \overline{\D^{k}}(M)$ be a closed $G$-invariant  set, and $c$ be a $G$-equivariant center on $A$. Then $c$ may be extended to a $G$-equivariant center  on  a $G$-invariant closed set which contains an open neighborhood of $A$.
\end{lem}
\begin{proof}
Each model plane $M$ admits a standard embedding $M\to\R^3$ that is equivariant
with respect to a representation $\Isom(M)\to GL(3)$. 
If $M=\S^2$, we have the inclusion map $\S^2\to\R^3$, and the identification of 
$\Isom(\S^2)$ with $O(3)\leq GL(3)$.  If $M=\R^2$, we identify $\R^2$ with the plane $\R^2\times\{1\}\subset\R^3$ and note that $\Isom(\R^2)$ is a subgroup of affine transformations of $\R^2$, which may be represented as a subgroup of $GL(3)$ via homogeneous coordinates. If $M=\H^2$, we consider the hyperboloid model of $\H^2$ in $\R^3$, 
so that $\Isom(\H^2)$ lies in the Lorentz group $O(1,2)\leq GL(3)$.

Thus $c$ may be viewed as a $G$-equivariant map
$c\colon A\to \R^3$. Now by Feragen's theorem \cite[Thm. 3.1]{feragen2010}
$c$ may be extended to a $G$-equivariant map $c\colon\overline{\D^{k}}(M)\to\R^3$, because $\overline{\D^{k}}(M)\subset \D(M)$ and $[\overline{\D^{k}}(M)]\subset[\D(M)]$ are metrizable  by Lemmas \ref{lem:metrizable} and \ref{lem:DandE} respectively. 
Let $U$ be the open $\epsilon$-neighborhood of $A$ in $\D(M)$ with respect to the metric $d$ given by \eqref{eq:metric}, and $\ol U$ be the closure of $U$. Since $d$ is $G$-invariant, 
$\ol U$ is $G$-invariant. Furthermore, we may choose $\epsilon$ so small that $c(\ol U)$ lies in a $G$-equivariant tubular neighborhood $\mathcal{T}$ of $M$ in $\R^3$. Then the $G$-equivariant projection $\pi\colon\mathcal{T}\to M$ composed with $c$ yields a $G$-invariant map $\ol c\colon \ol U\to M$. If $\epsilon$ is sufficiently small, we can make sure that $\ol c(D)\in\inte(D)$ for all $D\in \ol U$, as desired.

We note that $\mathcal{T}$ and $\pi$ may be constructed as follows. Fix $r\in (0,1)$ and let $p:=(x,y,z)\in\R^3$. 
If $M=\S^2$ or $\R^2\times\{1\}$, 
let $\mathcal{T}$ be the open $r$-neighborhood of $M$ in $\R^3$, with $\pi(p):=p/|p|$, and $\pi(p):=(x/z,y/z, 1)$ respectively. If $M$ is the hyperboloid model of $\H^2$, given by $Q(p):=z^2-x^2-y^2=1$ and $z>0$,
 let $\mathcal{T}$ be given by $1-r<Q(p)<1+r$ and $z>0$,
with $\pi(p):=p/\sqrt{Q(p)}$. 
\end{proof}

Applying the last two lemmas alternately,  we now establish the main result of this section, which concludes the proof of Theorem \ref{thm:main}:

\begin{proof}[Proof of Proposition \ref{prop:model}]
Let $A\subset\D(M)$ be a $G$-invariant closed set, and $c$ be a $G$-equivariant center on $A$. By Lemma \ref{lem:Ds}, $\D^0(M)=\overline{\D^0}(M)$. So $\D^0(M)$ is closed, which implies that $A\cap\D^0(M)$ is closed. Thus, 
by Lemma \ref{lem:D2}, we may extend $c$ to $\D^0(M)$.
Next,
by Lemma \ref{lem:last}, we may extend $c$ to a $G$-invariant closed set which contains an open neighborhood of $\overline{\D^0}(M)$ in $\overline{\D^{1}}(M)$. Then we may extend $c$ to $\D^{1}(M)$ by Lemma \ref{lem:D2}. Thus $c$ is now defined on $\D^0(M)\cup\D^{1}(M)$, which includes $\overline{\D^{1}}(M)$ by Lemma \ref{lem:Ds}. Again using Lemma \ref{lem:last} we may extend $c$ to a $G$-invariant closed subset of $\overline{\D^{2}}(M)$ which includes an open neighborhood of $\overline{\D^{1}}(M)$. Then we extend $c$ to all of $\overline{\D^{2}}(M)$ by Lemma \ref{lem:D2}. But $\overline{\D^{2}}(M)=\D(M)$ by Lemma \ref{lem:Ds}, which completes the proof.
\end{proof}

We also obtain a version of Theorem \ref{thm:main} for conformal transformations:

\begin{thm}\label{cor:conformal}
Let $G=\Conf(\R^2)$ and $A\subset\D(\R^2)$ be a closed $G$-invariant set. Then any $G$-equivariant center on $A$ may be extended to a $G$-equivariant center on $\D(\R^2)$.
\end{thm}
\begin{proof}
By Theorem \ref{thm:main}, $c$ may be extended to an $O(2)$-equivariant center on $\D(\R^2)$. For each $D\in\D(\R^2)$, let $R(D)$ be the radius and $x_0(D)$ be the center of the circumscribing circle of $D$. Define 
 $\rho_D\in G$ by
 $
 \rho_D(x):=(x-x_0(D))/R(D),
 $
 so that $\rho_D(D)$ is inscribed in $\S^1$. Then
 $$
 \ol c(D):=\rho_D^{-1}\circ c\circ\rho_D(D)
 $$ 
is a center on $\D(\R^2)$, which coincides with $c$ on $A$. We claim that $g(\ol c(D))=\ol c(g(D))$, for all $g\in G$ and $D\in\D(\R^2)$, which will complete the proof.  First assume that $g$ is a translation or a dilation, i.e., $g(x)=r x+b$ for some $r>0$ and $b\in\R^2$. Then one may quickly check that $\rho_{g(D)}\circ g=\rho_D$, which implies $\rho^{-1}_{g(D)}=g\circ\rho_D^{-1}$. Hence
 $$
 \ol c(g(D))=\rho_{g(D)}^{-1}\circ c\circ\rho_{g(D)}(g(D))=g\circ\rho_{D}^{-1}\circ c\circ\rho_{D}(D)=g(\ol c(D)),
 $$
 as desired. Next suppose that $g\in O(2)$. Again a simple computation shows that $\rho_{g(D)}\circ g=g\circ \rho_D$, which implies $\rho_{g(D)}^{-1}\circ g=g\circ\rho_D^{-1}$. So, since $c$ is $O(2)$-equivariant,
 $$
 \ol c(g(D))=\rho_{g(D)}^{-1}\circ c\circ g\circ\rho_{D}(D)=\rho_{g(D)}^{-1}\circ g\circ c\circ\rho_{D}(D)=g(\ol c(D)).
 $$
This completes the proof since every element of $G$ is composed of a translation, a dilation, and a rotation.
  \end{proof}
 
\begin{note}\label{note:conformal}
Theorem \ref{cor:conformal} holds trivially for $\H^2$, since $\Conf(\H^2)=\Isom(\H^2)$; however, it does not hold for $\S^2$. Indeed $\Conf(\S^2)$ is generated by inversions through circles.  In particular, if $D\in\D(\S^2)$ is a hemisphere, $\Sym(D)$ is generated by inversions through circles centered on $\partial D$, which act transitively on $D$. Hence $\Fix(D)=\emptyset$.
\end{note}

\section{Applications}\label{sec:application}
Here we develop some applications of  Theorem \ref{thm:main} and its conformal analogue, Theorem \ref{cor:conformal}, which illustrate the utility of centers when used in conjunction with Carath\'{e}odory's theorem (Lemma  \ref{lem:Riemann}). A Jordan domain $D\in\D(M)$ is \emph{round} if there is a point in $D$ which has constant distance from all points of $\partial D$.
Let $\mathcal{RD}(M)\subset\D(M)$ denote the space of round Jordan domains in $M$. 

\begin{prop}\label{cor:contract}
$\D(\R^2)$ admits a strong deformation retraction onto $\mathcal{RD}(\R^2)$ which is equivariant 
under $\Conf(\R^2)$. Furthermore, $\D(\R^2)$ is equivariantly contractible with respect to $O(2)\subset\Isom(\R^2)$.
\end{prop}
\begin{proof}
Let $G=\Conf(\R^2)$. By Theorem \ref{cor:conformal} there exists a $G$-equivariant center $c$ on $\D(\R^2)$. 
For every $D\in\D(\R^2)$ let $f_D:=f_{D, c(D), (1,0)}$ be as in Lemma \ref{lem:Riemann}. If $g\in G$, then $g(f_D(o))=g(c(D))=c(g(D))=f_{g(D)}(o)$ by equivariance of $c$. So
$f_{g(D)}^{-1}\circ g\circ f_D$ is a conformal transformation 
of $B^2$ that fixes $o$, and hence lies in $O(2)$. Thus
\begin{equation}\label{eq:gfD}
g(f_D(tB^2))=f_{g(D)}(tB^2)
\end{equation}
for all $t\in [0,1]$. In particular note that if $g$ is a translation, 
$f_{g(D)}^{-1}\circ g\circ f_D$ is the identity because it preserves 
the direction of $(1,0)$ by definition of $f_D$.
For $t\in (0,1]$ and $D\in\D(\R^2)$, define 
$
h_{t, D}\colon B^2\to\R^2
$ 
by 
$$
h_{t, D}(x):=f_D(o)+\frac{f_D(tx)-f_D(o)}{t},
$$
and set $h_0(x):=\underset{t\to 0}{\lim}\ h_{t, D}(x)=f_D(o)+(df_D)_o(x)$. 
Then $D_t:=h_{t, D}(B^2)$ is a homotopy between $D_1=D$ and 
$D_0=(df_D)_o(B^2)$.
If $D$ is a round domain with radius $r_D$, 
then $f_D(x)=r_Dx+c(D)$; so  $h_{t, D}=f_D$ which shows $D_t=D$, or $D_t$ fixes elements of $\mathcal{RD}(\R^2)$ as desired.

Now we check that $g(D_t)=(g(D))_t$ for all $g\in G$, i.e., $D_t$ is $G$-equivariant.
By continuity it is enough to consider $t\in (0,1]$.
If $g$ is linear, then by \eqref{eq:gfD},
\begin{eqnarray*}
g(D_t)&=&g(f_D(o))+\frac{g(f_D(tB^2))-g(f_D(o))}{t}  \\
          &=&f_{g(D)}(o)+\frac{f_{g(D)}(tB^2)-f_{g(D)}(o)}{t} \;\;=\;\; (g(D))_t.
\end{eqnarray*}
Recall that if $g$ is a translation by $b$, then $f_{g(D)}=g\circ f_D=b+f_D$.
Hence 
\begin{eqnarray*}
(g(D))_t&=&f_{g(D)}(o)+\frac{f_{g(D)}(tB^2)-f_{g(D)}(o)}{t}\\
&=&b+f_D(o)+\frac{b+f_{D}(tB^2)-b-f_{D}(o)}{t}\;=\;b+D_t\;=\;g(D_t).
\end{eqnarray*}
We conclude that $g(D_t)=(g(D))_t$ for all $g\in G$, since $G$ is generated by translations and linear conformal 
maps.

Finally note that $D_0=(df_D)_o(B^2)$ is convex, because $(df_D)_o$ is linear.
Let $B_D\in\mathcal{RD}(\R^2)$ be the round domain centered at $c(D)$ and with 
the same area as $D$. Then the mapping
$\lambda\mapsto(1-\lambda) D_0 +\lambda B_D$, where $\lambda\in [0,1]$ and $+$ is the Minkowski sum, 
is a $G$-equivariant homotopy of $D_0$ to $B_D$. Concatenating this homotopy with $D_t$ gives
a $G$-equivariant strong deformation retraction of $\D(\R^2)$ onto $\mathcal{RD}(\R^2)$ as desired. 
Furthermore, $(\lambda, B_D)\mapsto(1-\lambda) B_D +\lambda B^2$ is an $O(2)$-equivariant
deformation retraction of $\mathcal{RD}(\R^2)$ to $B^2$. Hence
$\mathcal{D}(\R^2)$ is $O(2)$-equivariantly contractible. 
\end{proof}

The last observation implies that the space of Jordan curves  in $\R^2$, i.e., continuous injective maps $\S^1\to\R^2$ modulo homeomorphisms of $\S^1$, is equivariantly contractible to circles, with respect to $\Conf(\R^2)$. For rectifiable curves, this contraction can also be performed via curve shortening flow and rescaling \cite{lauer2013}. Using the above proposition we now prove:

\begin{thm}\label{cor:contract2}
Let $M$ be a complete connected surface of constant curvature, and $G=\Isom(M)$. Then
$\D(M)$ admits a $G$-equivariant strong deformation retraction onto $\mathcal{RD}(M)$. In particular  if $M=\S^2$, then $\D(M)$ admits a $G$-equivariant strong deformation retraction onto the space of hemispheres. Furthermore if $M=\R^2$ or $\mathbf{H}^2$, then $\D(M)$ is equivariantly contractible with respect to $O(2)\subset\Isom(M)$.
\end{thm}
\begin{proof}
After replacing $M$ by its universal Riemannian cover, and $G$ by isometries of the covering space, we may assume that $M$ is a model plane.
The case of $\R^2$ is already covered by  Proposition \ref{cor:contract}, since $\Isom(M)\subset\Conf(M)$. We reduce the cases of $\S^2$ and $\H^2$ to that of $\R^2$ via the exponential map, as follows.

By Theorem \ref{thm:main} there exists a $G$-equivariant center $c$ on $\D(M)$.  For $D\in\D(M)$ let $f_D:=f_{D,c(D), u_0}$ be as in Lemma \ref{lem:Riemann}, where $u_0\in T_{c(D)} M$ is any unit vector. Note that, for $t\in[0,1]$, $f_D(t B^2)$ does not depend on $u_0$. Indeed if $u_0'\in T_{c(D)}M$ is another unit vector, and $f_D':=f_{D,c(D), u_0'}$,  then $f_D'=f_D\circ\rho$ for some $\rho\in O(2)$, since $f_D^{-1}\circ f_D'$ is a conformal transformation of $B^2$ which fixes $o$. Let $R_D\in\mathcal{RD}(M)$ be the largest round domain centered at $c(D)$ which is contained in $D$, and $T$ be the infimum of $t\in [0,1]$ such that $f_D(t B^2)$ intersects $\partial R_D$. Then  $[T,1]\ni t\mapsto D_t:=f_D(tB^2)$ deforms $D=D_1$ to a domain $D_T$ in $R_D$ which intersects $\partial R_D$. 

Using the exponential map $\exp_{c(D)}\colon T_{c(D)}M\to M$ we identify $D_T$ with a domain $\ol D:=\exp_{c(D)}^{-1}(D_T)$ in $T_{c(D)} M\simeq\R^2$, and $R_D$ with the round domain $\overline{R}:=\exp_{c(D)}^{-1}(R_D)$ centered at the origin $o\simeq c(D)$ of $\R^2$. By Proposition \ref{cor:contract}, there exists a strong $G$-equivariant deformation retraction of $\D(\R^2)$ onto $\mathcal{RD}(\R^2)$. Restricting this retraction to $\ol D$ yields a deformation $\ol D_t\in\D(\R^2)$, $t\in[0,1]$, with $\ol D_0=\ol D$ and $\ol D_1$ a round domain centered at $o$. Let  $\lambda(t)\in\R^+$ be the dilation factor such that $\lambda(t)\ol D_t$ lies  inside $\ol R$ and intersects $\partial\ol R$. Then $\lambda(t) \ol D_t$ is a deformation of $\ol D$ to $\ol R$ through domains which are inside $\ol R$ and intersect $\partial\ol R$. So $\exp_{c(D)}(\ol D_t)$ is a deformation of $D_T$ to $R_D$ through domains which lie in $R_D$ and intersect $\partial R_D$. Concatenating this deformation with the earlier deformation of $D$ to $D_T$ yields a deformation of $D$ to $R_D$ which depends continuously on $D$, and keeps $D$ fixed whenever $D$ is round. Hence we obtain a strong deformation retraction $\D(M)\to\mathcal{RD}(M)$. 

To see that the retraction we have constructed is $G$-equivariant, note that $R_{\rho(D)}=\rho(R_{D})$ for all $\rho\in G$, which yields that $(\rho(D))_t=\rho(D_t)$. In particular $(\rho(D))_T=\rho(D_T)$, which yields that $\overline{\rho(D)}\subset T_{\rho(c(D))}M$ and $\ol D\subset T_{c(D)}M$ are  isometric, since $d\rho_{c(D)}(\overline{D})=\overline{\rho(D)}$.
It follows that the deformations $(\overline{\rho(D)})_t$ and $\ol D_t$ also correspond via the isometry $d\rho_{c(D_t)}\colon T_{c(D_t)}M\simeq\R^2\to T_{\rho(c(D_t))}M\simeq\R^2$, since $(\overline{\rho(D)})_t$ and $\ol D_t$ are constructed using the retraction given by Proposition \ref{cor:contract}, which is equivariant under $\Isom(\R^2)$. Thus we obtain a deformation of $\rho(D)$ to $R_{\rho(D)}$ which corresponds via $\rho$ to the deformation of $D$ to $R_D$ at each instant.

If $M=\H^2$, concatenating the homotopy which deforms $D$ to $R_D$ with translations along geodesic segments which connect $c(D)$ to a fixed point $o\in M$, followed by a dilation, yields a contraction of $\D(M)$ to the round domain $D_0$ of radius $1$ centered at $o$. This yields a contraction of $\D(\H^2)$ to $D_0$ which is equivariant under $O(2)\subset\Isom(\H^2)$. If $M=\S^2$, we may dilate  each domain $R_D$ with respect to $c(D)$ until radius of $R_D$ reaches $\pi/2$, which constitutes the desired retraction of $\D(\S^2)$ onto hemispheres, and competes the proof.
\end{proof}

\section{Canonical  Centers}\label{sec:canonical}

The centers that we constructed in Section \ref{sec:equiv} involved arbitrary choices for the extensions of prescribed values across each stratum of $\D(M)$. Here we describe a procedure for constructing centers on smooth Jordan domains in $\R^2$, which correspond in a canonical way to the center of mass, Steiner point, or circumcenter (the center of the circumscribing circle) of convex domains. In this section we assume that $G=\Conf(\R^2)$.

\begin{lem}\label{lem:projection}
For any $D\in \D(\R^2)$, there exists a strong deformation retraction $ r_D\colon \R^2\times[0,1]\to\R^2$  of $\R^2$ onto $D$ such that the mapping
$$
\D(\R^2)\;\;\ni\;\; D\overset{r}{\longmapsto}  r_D\;\;\in\;\;\C^0\big(\R^2\times[0,1],\R^2\big)
$$
is continuous and $G$-equivariant, i.e., $\rho(r_D(x,t))= r_{\rho(D)}(\rho(x),t)$ for all $\rho\in G$ and $D\in\D(\R^2)$.
\end{lem}
\begin{proof}
As in the proof of Theorem \ref{cor:conformal}, for $D\in\D(\R^2)$, let $x_0(D)$ be the circumcenter of $D$, $R(D)$ be the radius of the circumscribing circle, and set $ \rho_D(x):=(x-x_0(D))/R(D)$.
Then the mapping $D\mapsto \rho_D(D)$ is a retraction from $\D(\R^2)$ onto the space $\D_0(\R^2)$ of domains with $x_0(D)=o$ and $R(D)=1$. By a discussion similar to the proof of Theorem \ref{cor:conformal}, it is enough to construct the retractions $r_D$ that we seek for  $D\in \D_0(\R^2)$, and show only that they are $O(2)$-equivariant. Then 
$
r_D(x,t):=\rho_D^{-1}\circ r_{\rho(D)}\big(\rho_D(x),t\big)
$
yields the desired retraction for all $D\in\D(\R^2)$. 

So assume that $D\in \D_0(\R^2)$. Let $D^*$ be the closure of $\R^2\setminus D$, $p_0:=(0,0,1)$, $st\colon\S^2\setminus\{p_0\}\to\R^2$ be the stereographic projection, and $\tilde{D^*}\subset\S^2$ be the closure of $st^{-1}(D^*)$. Furthermore, fix $u_0\in T_{p_0}\S^2\setminus\{0\}$, and  let $f_{\tilde{D^*}}:=f_{\tilde{D^*},p_0,u_0}\colon B^2\to\tilde{D^*}$ be as in Lemma \ref{lem:Riemann}. Then we obtain a conformal map $g_D\colon B^2\setminus\{o\}\to D^*$ given by
$
g_D:=st\circ f_{\tilde{D^*}}.
$
Next let $h_t\colon B^2\setminus\{o\}\to B^2$, given by $h_t(x):=(1-t)x +t x/|x|$ be the strong deformation retraction of $B^2\setminus\{o\}$ onto $\S^1$, and set
$$
 r_D(x,t):=g_D\circ h_t\circ g^{-1}_D (x).
$$

By Lemma \ref{lem:Riemann}, $D\mapsto  r_D$ is continuous. 
It remains to check that $ r_D$ is equivariant with respect to $\rho\in O(2)$. So we compute that
\begin{eqnarray*}
\rho(r_D(x,t))&=&\rho\circ g_D\circ h_t\circ g^{-1}_D (x)\\
&=& 
g_{\rho(D)}\circ \big(g_{\rho(D)}^{-1}\circ \rho\circ g_D\big)\circ h_t\circ \big(g_{\rho(D)}^{-1}\circ \rho\circ g_D\big)^{-1}\circ g_{\rho(D)}^{-1}\circ\rho(x)\\
&=& 
g_{\rho(D)}\circ  h_t\circ g_{\rho(D)}^{-1}\circ\rho(x)\;\;=\;\;r_{\rho(D)}(\rho(x),t).
\end{eqnarray*}
The third equality above holds because $g_{\rho(D)}^{-1}\circ \rho\circ g_D$ is a conformal transformation of $B^2\setminus\{o\}$, so it is a rotation  about $o$, which commutes with $h_t$.
\end{proof}

An analogue of the above result for curves had been established earlier by Pixley \cite{pixley1976}; see also \cite{repovs-semonov2010}. Now any continuous map $f\colon \D(\R^2)\to\R^2$ generates a continuous point selection on $\D(\R^2)$ given by $r_D(f(D),1)$, which is $G$-equivariant whenever $f$ is $G$-equivariant. To obtain a center we need to push this point into the interior of $D$ in a continuous way. 
We show how to do this on the space of $\C^2$ Jordan domains $\D_{\C^2}(\R^2)$, i.e., images of $\C^2$ embeddings $B^2\to \R^2$ equipped with $\C^2$-topology. So a pair of domains in $\D_{\C^2}(\R^2)$ are close provided that they admit parametrizations that are $\C^2$-close.
 For each $D\in \D_{\C^2}(\R^2)$ let $\textup{reach}(D)$ denote its \emph{reach} in the sense of Federer \cite{federer:curvature,thale2008}. More explicitly,   $\textup{reach}(D)$ is the supremum of numbers $r$ such that  through every point of $\partial D$ there passes a circle of radius $r$ which is contained in $D$. Reach may also be defined as the distance between $\partial D$ and the \emph{medial axis} of $D$, i.e., the set of singularities of the distance function from the boundary $d_{\partial D}\colon D\to\R$, see \cite[Sec. 3]{ghomi-spruck2022}. 
Since $\partial D$ is $\C^2$, it follows that $d_{\partial D}$ is $\C^2$ within the open neighborhood of $\partial D$ of radius $\textup{reach}(D)$ \cite{foote1984}, and hence $\textup{reach}(D)>0$. Furthermore, Chazal and Soufflet \cite[Thm. 3.2]{chazal-soufflet2004} showed that the medial axis of $D$ varies continuously, with respect to Hausdorff distance, under $\C^2$ perturbations of $\partial D$. So we may record that:

\begin{lem}\label{lem:reach}
For any domain $D\in \D_{\C^2}(\R^2)$, $\textup{reach}(D)>0$. Furthermore, the mapping $D\mapsto \textup{reach}(D)$ is continuous on $ \D_{\C^2}(\R^2)$.
\end{lem}

Let $D^{\frac12}\subset D$ be the domain obtained by moving $\partial D$  along its inward normals by the distance $\textup{reach}(D)/2$. So $D^{\frac12}$ is characterized by
\begin{equation}\label{eq:Dlambda}
D=D^{\frac12} + \frac{\textup{reach}(D)}{2} B^2,
\end{equation}
where $+$ is the Minkowski sum. Alternatively, $D^\frac12$ may be defined as the set $\{d_{\partial D}\geq \textup{reach}(D)/2\}$. Note that $d_{\partial D}$ is a convex function if $D$ is convex \cite[Lem.  3.3, p. 211]{sakai1996}. Thus $D^\frac12$ is convex whenever $D$ is convex. Since $d_{\partial D}$ is $\C^2$ within the distance $\textup{reach}(D)$  of $\partial D$, we have
$D^\frac12\in \D_{\C^2}(\R^2)$. By Lemma \ref{lem:reach}, $D\mapsto D^{\frac12}$ is continuous, and $D^\frac12\subset\inte(D)$. Furthermore, since conformal transformations are affine, and affine transformations preserve Minkowski addition, \eqref{eq:Dlambda} shows that $D\mapsto D^{\frac12}$ is $G$-equivariant. 
Thus any continuous map $f\colon  \D_{\C^2}(\R^2)\to\R^2$ generates a center on $\D_{\C^2}(\R^2)$ given by
\begin{equation}\label{eq:cf}
c_f(D):=r_{\!\!D^\frac12}\big(f(D),1\big),
\end{equation}
 where $r$ is as in Lemma \ref{lem:projection}. Note that $c_f$ is $G$-equivariant whenever $f$ is $G$-equivariant, and $c_f(D)=f(D)$ whenever $f(D)\in D^\frac12$. The next lemma gives examples of $f$ with this property. For $D\in\D_{\C^2}(\R^2)$, 
the \emph{Steiner point} of $D$ is the center of mass of $\partial D$ with respect to the density function given by its curvature. 

\begin{lem}
If $f\colon \D_{\C^2}(\R^2)\to\R^2$ is the Steiner point, circumcenter, or center of mass, then $f(D)\in D^\frac12$ whenever $D$ is convex.
\end{lem}
\begin{proof}
First assume that $f$ is the 
Steiner point. Then $f$ is additive with respect to Minkowski sum for convex domains \cite{schneider1971, schneider2014}. Hence when $D$ is convex, 
$$
f(D)=f\Big(D^{\frac12} + \frac{R}{2} B^2\Big)=f(D^{\frac12} )+\frac{R}{2} f(B^2)=f(D^{\frac12})\in D^\frac12,
$$
where $R:=\textup{reach}(D)$.
Next suppose that $f$ is the circumcenter. Note that the circumscribing circle of $D$ is concentric to that of $D^\frac12$. So again we have
$
f(D)=f(D^\frac12),
$
 which yields $f(D)\in D^\frac12$ when $D$ is convex. Finally let $f$ be the center of mass. Set $A:=D\setminus D^\frac12$. Then
 $$
 f(D)=\frac{|D^\frac12|}{|D|}f(D^\frac12)+\frac{|A|}{|D|}f(A),
 $$
 where $|\cdot|$ denotes area. Thus it suffices to show that $f(A)\in D^\frac12$, because $D^\frac12$ is convex. Let $\gamma\colon[0,L]\to\R^2$ be an arc-length parametrization for $\partial D^\frac12$, $N(t)$ be the inward unit normal vector along $\gamma$, and set
 $
 x(t,s):=\gamma(t)-sN(t).
 $
 Then 
 $$
 f(A)=\frac{1}{|A|}\int_0^\frac{R}{2}\int_0^L x(t,s)\left|\frac{\partial x}{\partial t}\times \frac{\partial x}{\partial s}\right| dt\,ds.
$$
Set $T(t):=\gamma'(t)$. Then $N'(t)=-\kappa(t)T(t)$, where $\kappa:=|T'|$ is the curvature of $\gamma$. A simple computation shows that
$
\left|\frac{\partial x}{\partial t}\times \frac{\partial x}{\partial s}\right|=1+s\kappa(t).
$
Next note that $T'=\kappa N$ and $\int_0^L T'(t)\,dt=T(L)-T(0)=0$. Furthermore, if $\gamma$ parametrizes $\partial D$ in the counterclockwise direction, then $N(t)=i T(t)$, where $i$ indicates counterclockwise rotation by $\pi/2$. So
$\int_0^L N(t)\,dt=i\int_0^L T(t)\,dt=i(\gamma(L)-\gamma(0))=0$. Now we have
\begin{gather*}
\int_0^\frac{R}{2}\int_0^L x(t,s)\left|\frac{\partial x}{\partial t}\times \frac{\partial x}{\partial s}\right| dt \,ds\\
=\int_0^\frac{R}{2}\int_0^L\big(\gamma(t)+s\gamma(t)\kappa(t) -sN(t)-s^2T'(t)\big)dt \,ds\\
=\int_0^\frac{R}{2}\int_0^L \gamma(t)\big(1+s\kappa(t)\big)dt\,ds.
\end{gather*}
So we may write
$$
f(A)=\int_0^L\gamma(t)\delta(t)dt, \quad\quad\text{where}\quad\quad \delta(t):=\frac{1}{|A|}\int_0^\frac{R}{2}\big(1+s\kappa(t)\big)ds.
$$
But $\int_0^L \delta(t)dt=|A|/|A|=1$. Hence $f(A)$ is the center of mass of $\partial D^\frac12$ with respect to the density function $\delta$, which yields that $f(A)\in D^\frac12$ by \cite[Lem. 2.3]{ghomi:knots}.
\end{proof}
Now we may conclude that:

\begin{thm}\label{thm:C2}
Any continuous $G$-equivariant map $f\colon \D_{\C^2}(\R^2)\to\R^2$ generates a $G$-equivariant center $c_f$ on $\D_{\C^2}(\R^2)$ given by \eqref{eq:cf}. Furthermore, if $f$ is the Steiner point, center of mass, or circumcenter, then $c_f(D)=f(D)$ whenever $D$ is convex.
\end{thm}

The obvious question which remains is how to extend the last result to construct centers for $\D(\R^2)$, and more generally for $\D(M)$, which correspond to classical notions of center in a canonical way.

\begin{note}
Let $\D_{\C^\omega}(\R^2)\subset \D_{\C^2}(\R^2)$ denote the space of planar Jordan domains with analytic boundary. Another approach to construct a center on $\D_{\C^\omega}(\R^2)$ is as follows. The medial axis of a domain $D\in \D_{\C^\omega}(\R^2)$ forms a piecewise analytic tree $T(D)$  \cite{choi-choi-moon1997}. In particular $T(D)$ is rectifiable, and we may define $c(D)$ as the center of $T(D)$ \cite{bandelt1992}, i.e., the point in $T(D)$ whose maximum distance from leaves of $T(D)$, measured in $T(D)$, is as small as possible. 
By Lemma \ref{lem:reach}, $c(D)\in\inte(D)$. Furthermore, the continuity of $c(D)$ should follow from \cite[Thm. 7.2]{choi-choi-moon1997} or  \cite[Thm. 3.2]{chazal-soufflet2004}. See \cite{abe2009} for a survey of literature on medial axis and its stability. For an introduction to medial axis on manifolds and  more references see \cite[Sec. 3]{ghomi-spruck2022}.
\end{note}

\bibliographystyle{abbrv}
\bibliography{references}

\end{document}